\documentclass[12pt]{amsart}
\usepackage[pagewise]{lineno}
\usepackage{amscd,amssymb}
\usepackage[arrow,matrix]{xy}
\usepackage[plainpages,backref,urlcolor=blue]{hyperref}
\usepackage{mathrsfs}
\usepackage{amsthm}
\usepackage{bm}
\usepackage{amscd}
\usepackage[shortlabels]{enumitem}
\usepackage{mathdots}
\usepackage{mathrsfs}
\usepackage{graphics}
\usepackage{mathtools}

\theoremstyle{plain}
\newtheorem{thm}[subsection]{Theorem}
\newtheorem{lem}[subsection]{Lemma}
\newtheorem{prop}[subsection]{Proposition}

\newtheorem{claim}[subsection]{Claim}

\theoremstyle{definition}
\newtheorem{rk}[subsection]{Remark}
\newtheorem{definition}[subsection]{Definition}

\newtheorem{ex}[subsection]{Example}

\newcommand{\ov}{\overline}
\newcommand{\bb}{\mathbb}

\newcommand{\qe}{\qquad\hfill \square}

\newcommand{\Mod}[1]{\,\text{\rm mod}\,#1}

\newcommand{\Sing}{\text{\rm Sing}}

\newcommand{\bg}[1]{\mathbf{#1}}

\begin{document}
	\date{}
		
	\title[On algebraic surfaces associated to line arrangements]{ On algebraic surfaces associated to line arrangements}
		
	\author[ZHENJIAN WANG]{ ZHENJIAN WANG  }
	\address{Universit\'e C$\hat{\rm o}$te d'Azur, CNRS,  LJAD, UMR 7351, 06100 Nice, France.}
	\email{wzhj01@gmail.com}
		
		
		
	\begin{abstract}
	For a line arrangement in the complex projective plane $\bb{P}^2$, we investigate the compactification $\ov{F}$ of the affine Milnor fiber in $\bb{P}^3$ and its minimal resolution $\widetilde{F}$. We compute the Chern numbers in terms of the combinatorics of the line arrangement, then we show that the minimal resolution is never a quotient of a ball; in addition, we also prove that $\widetilde{F}$ is of general type when the arrangement has only nodes or triple points as singularities.
	\end{abstract}
	\maketitle
   \tableofcontents

\section{Introduction}
A {\bf ball quotient} is a smooth projective surface that is biholomorphic to $\mathbf{B}/\Gamma$, where
$$
\mathbf{B}=\{(z,w)\in\bb{C}^2: |z|^2+|w|^2<1\}
$$
equipped with the K\"ahler metric whose K\"ahler form is given by
$$
\omega_{\mathbf{B}}=-\sqrt{-1}\partial\ov{\partial}\log(1-|z|^2-|w|^2),
$$
 and $\Gamma$ is a discrete cocompact subgroup of isometries of ${\bf B}$. A smooth projective surface $X$ is called of {\bf general type} if its canonical divisor $K_X$ is big; and for such surfaces, we have the celebrated Miyaoka-Yau inequality, namely,
\begin{equation}\label{eq: MYineq}
c_1^2(X)\leq 3c_2(X),
\end{equation}
where the equality holds if and only if $X$ is a ball quotient, see \cite{Yau}.

Let $\mathcal{A}=\{L_1,\cdots,L_d\}$ be a line arrangement in the complex projective plane $\bb{P}^2$, where $L_i: \ell_i(x,y,z)=0,\ i=1,\cdots, d$ and let $Q=\ell_1\ell_2\cdots \ell_d$ be the defining polynomial of $\mathcal{A}$. In \cite{HI}, F. Hirzebruch considers the Kummer covers of $\bb{P}^2$ branched over the arrangement $\mathcal{A}$, and by desingularization, he obtains smooth algebraic surfaces of general type; in addition, he also shows that for some line arrangements, the associated surfaces are ball quotients. Later, Hirzebruch's method was extended to a more general setup, see for instance \cite{Ma}.

In this article, we consider another construction of surfaces associated to the arrangement $\mathcal{A}$. Let $F: Q=1$ be the affine {\bf Milnor fiber} in $\bb{C}^3$ of $\mathcal{A}$, $V(Q): Q=0$ be the union of lines contained in $\mathcal{A}$ and $M=\bb{P}^2\setminus V(Q)$ be the complement of $\mathcal{A}$. Then there is a natural Galois covering map $\rho: F\to M$ of degree $d$. Let $h: F\to F$ be given by $h(x)=\exp(2\pi\sqrt{-1}/d)\cdot x$ be the multiplication by a primitive $d$-root of unity. As is shown in \cite{DP11}, $H^1(F)$ admits a mixed Hodge structure (MHS) with only two two weights: 1 and 2, for which the induced morphism by $h$, namely, $h^*: H^1(F)\to H^1(F)$ is a morphism of MHSs. Moreover, in \cite{DP11}, the authors show that
$$
W_1H^1(F)=H^1(F)_{\neq 1}
$$
and
$$
Gr_2^WH^1(F)=H^1(F)_1=\rho^*H^1(M),
$$
where
$$
\quad H^1(F)_{\neq 1}=\ker(h^{*d-1}+\cdots+h^*+\text{Id})
$$
while
$$
H^1(F)_1=\ker(h^*-\text{\rm Id}).
$$
Note that $\dim Gr_2^WH^1(F)=\dim H^1(M)=d-1$ depending only on the number of lines in the arrangement $\mathcal{A}$. For many known examples, $\dim W_1H^1(F)=\dim H^1(F)_{\neq1}$, which depends on the monodromy $h$, is very small: for instance, for Hesse arrangement, $\dim H^1(F)_{\neq 1}$ is 6 (see \cite{BDS}) and for the arrangement $\mathcal{A}(m,m,3)$ in \cite{Dpre1}, this number is at most 4. The interested reader may find more such examples in \cite{Dpre1} and \cite{CS95}.

To explore the deep reasons for the smallness of $\dim H^1(F)_{\neq 1}$ and also try to find more ball quotients, we consider the natural compactification of $F$ in $\bb{P}^3$, namely,
$$
\ov{F}\quad:\quad Q(x,y,z)+t^d=0.
$$
By resolving the singularities of $\ov{F}$, we obtain a smooth projective surface. Let $\pi: \widetilde{F}\to\ov{F}$ be the {\bf minimal resolution} of $\ov{F}$. We say that $\mathcal{A}$ is a \emph{pencil} if $V(Q)$ has a singularity of multiplicity $d=|\mathcal{A}|$.  We prove the following.

\begin{thm}\label{thm: not ball quotient}
Assume $d=|\mathcal{A}|\geq 2$. Then
\begin{enumerate}[$\rm (i)$]
\item if $d\geq 4$ and $\mathcal{A}$ is a pencil, $c_1^2(\widetilde{F})>3c_2(\widetilde{F})$.
\item if $d\geq3$ and $\mathcal{A}$ is not a pencil or $d=2$, $c_1^2(\widetilde{F})<3c_2(\widetilde{F})$.
\end{enumerate}

In particular, if $d=2$ or $d\geq 3$ and $\mathcal{A}$ is not a pencil, then $\widetilde{F}$ is not a ball quotient.
\end{thm}

The non-ball-quotient property of $\widetilde{F}$ is always true for $d\geq 2$, see Remark \ref{rk: exceptional case} below. In fact, if $\mathcal{A}$ is not a pencil, $\widetilde{F}$ contains several rational curves according to Theorem \ref{thm: resolhom} below, thus it is not surprising that $\widetilde{F}$ is not a ball quotient because a ball quotient cannot contain any rational curve, see \cite{Cat}, Proposition 19.

However, our results above give much preciser numerical properties about the Chern numbers. In particular, $c_1^2(\widetilde{F})<3c_2(\widetilde{F})$ holds when $\mathcal{A}$ is not a pencil. So it is natural to discuss whether $\widetilde{F}$ is a surface of general type. We prove the following result in Section \ref{sec: general type}.

\begin{thm}\label{thm: general type}
Assume $d=|\mathcal{A}|\geq 7$ and $V(Q)$ contains only nodes or triple points as singularities, then $\widetilde{F}$ is of general type.
\end{thm}

In fact, the Chern numbers of $\widetilde{F}$ are uniquely determined by the combinatorics of $\mathcal{A}$. Let $t_r$ be the number of singular points in $V(Q)$ of multiplicity $r$. Then we have the following.

\begin{thm}\label{thm: chern}
Let $\mathcal{A}$ be a line arrangement in $\bb{P}^2$ consisting of $d=|\mathcal{A}|$ lines. Then
\begin{enumerate}[\rm (i)]
\item the first Chern number of the associated surface $\widetilde{F}$ is given by
$$
c_1^2(\widetilde{F})=K_{\ov{F}}^2+\sum_rt_r DCI_{r,d}
$$
where $K_{\ov{F}}^2=d(d-4)^2$ and for $d\not\equiv1\Mod r$, we have
$$
DCI_{r,d}=-d(r-2)^2-r\sum_{i=1}^\lambda(n_i-2)+2(r-2)(r-\gcd(r,d))+(r-b);
$$
for $d\equiv1\Mod r$, we have
$$
DCI_{r,d}=-(d-1)(r-2)^2.
$$

\item the second Chern number of the associated surface$\widetilde{F}$ is given by
$$
c_2(\widetilde{F})=\chi(\ov{F})+\sum_rt_r DCII_{r,d}
$$
where $\chi(\ov{F})$ is the topological Euler number of $\ov{F}$
$$
\chi(\ov{F})=d(d^2-4d+6)-(d-1)\sum_rt_r(r-1)^2.
$$
In addition, for $d\not\equiv1\Mod r$, we have
$$
DCII_{r,d}=1+r\lambda-(r-2)(\gcd(r,d)-1);
$$
for $d\equiv1\Mod r$, we have
$$
DCII_{r,d}=d-1.
$$
\end{enumerate}
In the above formulae, the numbers $\lambda,b,n_i$'s are uniquely determined by $r$ and $d$ only  from Theorem \ref{thm: resolhom} below.
\end{thm}

One of motivation of our work is to understand whether the Hodge numbers of $\widetilde{F}$ are combinatorially determined, one of the main open question in the theory of line arrangements, see \cite{PS14}. As is explained in Section \ref{sec: chern and hodge}, we have the following formulae
\begin{equation*}
\begin{cases}
h^{0,0}&=h^{2,2}=1,\\
h^{0,1}&=h^{1,0}=h^{1,2}=h^{2,1}=q,\\
h^{0,2}&=h^{2,0}=\frac{1}{12}(c_1^2+c_2)-(1-q),\\
h^{1,1}&=-\frac{1}{6}c_1^2+\frac{5}{6}c_2+2q.
\end{cases}
\end{equation*}
where $h^{p,q}$'s denote the Hodge numbers for a given smooth projective surface, $c_1^2, c_2$ denote the Chern numbers and $q$ the irregularity. For the associated surface $\widetilde{F}$, the Chern numbers $c_1^2, c_2$ are determined by the combinatorics of $\mathcal{A}$ by Theorem \ref{thm: chern}. On the other hand, it follows from \cite{DP11} that $2q=\dim H^1(F)_{\neq 1}$, which is known to many line arrangements, see \cite{CS95}. In fact, in \cite{PS14}, a combinatorial formula for $q$ is given when $\mathcal{A}$ has only double or triple points; more examples are given in \cite{S1} where $q$ is computed. \cite{S4} is a good and recent survey on the monodromy computations and in a recent preprint \cite{DStgenMF}, an effective algorithm to compute $q$ is provided.

To illustrate more numerical properties of the surface $\widetilde{F}$, we will give some examples in which we compute all the Hodge numbers of $\widetilde{F}$ in the end of this paper. We also compute the Chern ratio. This numerical invariant is of special interest for algebraic surfaces, see for instance \cite{MP},\cite{MPP},\cite{Naie},\cite{Rou}.

Note that in \cite{HI}, some combinatorial inequalities about the number of multiple points of line arrangement $\mathcal{A}$ were obtained by applying the Miyaoka-Yau inequality; these inequalities also play a key role in our proof of Theorem \ref{thm: not ball quotient}.

In our situation, it still remains a question on how to decide whether $\widetilde{F}$ is of general type in a more general situation; it also remain open whether $\widetilde{F}$ is a minimal surface, i.e., whether $\widetilde{F}$ contains $(-1)$-curves that are not contracted by $\pi$. All these issues will be addressed in subsequent papers.

\section{General setting}
In this section, we first present some basic facts about normal (not necessarily smooth) surfaces that will be used in the sequel. Although none of them is new, it is hard to find a good reference. Then we focus on surfaces associated to line arrangements.

\subsection{Intersection theory for normal surfaces}
Let $X$ be a projective variety of dimension $n$ over $\bb{C}$. When $X$ is smooth, then the intersection theory on $X$ is classical and quite well-known. But if $X$ is not smooth, there are technical problems about defining the intersection number of two divisors on $X$, see for instance \cite{EH}, Section 5.5 and also \cite{Fu}. However, we can always have a well-defined intersection number of $n$ {\bf Cartier divisors} on the projective variety $X$ and such an intersection theory admits similar properties as in the smooth case, see \cite{Deb}, Chapter 3.
Notation: the intersection number of $n$ Cartier divisors $D_1,\cdots, D_n$ will be denoted by $D_1\cdot\cdots\cdot D_n$; if $D_1=\cdots=D_n=D$, the intersection number is denoted by $D^n$.

In addition, given a morphism $f: C\to X$ from projective curve to a quasi-projective variety, and $D$ a Cartier divisor (class) on $X$, we define
\begin{equation}\label{eq: curveCartier}
D\cdot C=\deg(f^*D).
\end{equation}
We will mainly be concerned with intersection theory on normal surfaces.

Let $X,\widetilde{X}$ be two normal projective surfaces and $\pi:\widetilde{X}\to X$ a proper surjective morphism. Then we have the {\bf pull-back formula}
\begin{equation}\label{eq: pullback}
\pi^*C\cdot\pi^*D=C\cdot D
\end{equation}
for any two Cartier divisors on $X$, see \cite{Deb}, Proposition 3.16. Moreover, the pull-back formula \eqref{eq: pullback} together with Formula \eqref{eq: curveCartier} implies the {\bf projection formula}:
\begin{equation}\label{eq: projection formula}
\pi^*D\cdot E=D\cdot\pi_*E,
\end{equation}
where $D$ is a Cartier divisor on $X$ while $E$ is a curve on $\widetilde{X}$. In particular, if $E$ is contracted by $\pi$ to a point, then $\pi^*D\cdot E=0$ for any Cartier divisor $D$ on $X$.

\subsection{Canonical divisors of normal surfaces}
In the sequel of this section, let $X$ be a normal surface on a smooth projective threefold $Y$. If $X$ is smooth, then we have a well-defined canonical bundle and hence the canonical divisor $K_X$; moreover, we have the {\bf adjunction formula}
\begin{equation}\label{eq: X-adjunction}
K_X=(K_Y+X)|_X.
\end{equation}
When $X$ is not smooth, then we have a canonical bundle on the smooth locus $X\setminus\text{\rm Sing}(X)$ of $X$, and hence the associated Cartier divisor $K_{X\setminus\text{\rm Sing}(X)}$ on $X\setminus\text{\rm Sing}(X)$. The canonical divisor of $X$, still denoted by $K_X$, is the closure in $X$ is $K_{X\setminus\text{\rm Sing}(X)}$ and $K_X$ is a Weil divisor on $X$. Furthermore, the adjunction formula \eqref{eq: X-adjunction} still holds. Indeed, the equality clearly holds on the smooth locus $X\setminus\text{\rm Sing}(X)$; but $\text{\rm Sing}(X)$ has codimension 2 in $X$ since $X$ is normal, it follows that any Weil divisor on $X$ is uniquely determined by its restriction on $X\setminus\text{\rm Sing}(X)$, and thus Formula \eqref{eq: X-adjunction} is valid on $X$.

Note that $Y$ being smooth, it follows that $K_Y, X$ are both Cartier divisors on $Y$, hence $K_X=(K_Y+X)|_X$ is a Cartier divisor on $X$, and we have the intersection number $K_X^2$. By \cite{Deb}, Proposition 3.15, we obtain
\begin{equation}\label{eq: KX2}
K_X^2=(K_Y+X)\cdot(K_Y+X)\cdot X.
\end{equation}

\subsection{Miyaoka-Yau number}
 When $X$ is smooth, then $c_1^2(X)=K_X^2$ and $c_2(X)=\chi(X)$, the topological Euler number. If $X$ is not smooth, $K_X^2$ and $\chi(X)$ are still well-defined numbers for $X$. In view of the Miyaoka-Yau inequality, we give the following definition.
\begin{definition}
Let $X\subseteq Y$ be a normal surface on a smooth projective threefold $Y$, the \emph{Miyaoka--Yau number} of $X$ is defined by
$$
MY(X)=3\chi(X)-K_X^2.
$$
\end{definition}

Let the normal surface $X$ have an isolated singularity $0\in X$, and $\pi:\widetilde{X}\to X$ be a minimal resolution of the singularity 0, given by successive embedded blowups. Let $\pi':\widetilde{Y}\to Y$ be the effect of the successive blowups on $Y$. Then $\widetilde{Y}$ is a smooth projective threefold, on which $\widetilde{X}$ is a normal surface, hence we have the canonical divisor $K_{\widetilde{X}}=(K_{\widetilde{Y}}+\widetilde{X})|_{\widetilde{X}}$ and the Miyaoka-Yau number of $\widetilde{X}$:
$$
MY(\widetilde{X})=3\,\chi(\widetilde{X})-K_{\widetilde{X}}^2.
$$

\begin{definition}
Three numerical invariant differences for the minimal resolution $\pi:\widetilde{X}\to X$ are defined as follows:
\begin{enumerate}[(i)]
\item The difference for the first Chern number is
$$
DCI=K_{\widetilde{X}}^2-K_X^2;
$$
\item The difference for the second Chern number is
$$
DCII=\chi(\widetilde{X})-\chi(X);
$$
\item The difference for the Miyaoka-Yau number is
$$
DMY=MY(\widetilde{X})-MY(X)=3\,DCII-DCI.
$$
\end{enumerate}
\end{definition}

When $X$ is given by $(X,0): G_r(u,v)+t^d=0$ around the local coordinates $(u,v,t)$ centered at 0 on $Y$ where $G_r$ is a product of $r$ distinct linear forms, it turns out that the three differences defined above are determined by $r$ and $d$, and thus will be denoted by $DCI_{r,d}, DCII_{r,d}$ and
$
DMY_{r,d}=3\,DCII_{r,d}-DCI_{r,d}.
$

\section{Surfaces associated to line arrangements}

Let $\mathcal{A}=\{L_1,\cdots, L_d\}$ with $L_i:\ell_i=0, i=1,\cdots, d$, be a line arrangement in $\bb{P}^2$ with defining polynomial $Q(x,y,z)=\ell_1\ell_2\cdots\ell_d$.

Given $r\geq 2$. If a point $x\in\bb{P}^2$ lies on exactly $r$ lines in $\mathcal{A}$, or equivalently, $x$ is a singular point of multiplicity $r$ of the curve $V(Q): Q=0$ in $\bb{P}^2$, we say that $x$ is of multiplicity $r$.
The number of points of multiplicity $r$ will be denoted by $t_r$.

Consider the affine Milnor fiber $F: Q=1$ in $\bb{C}^3$, for which we have a natural compactification
$$
\ov{F}\quad:\quad Q(x,y,z)+t^d=0
$$
in $\bb{P}^3$. $\ov{F}$ is a singular normal surface in $\bb{P}^3$; a singular point of multiplicity $r$ of $V(Q)$ gives a singular point of multiplicity $r$ of $\ov{F}$, and vice versa. Moreover, since $Q$ is a product of distinct linear forms, around a singular point of $\ov{F}$ of multiplicity $r$, we have $\ov{F}: G_r(u,v)+t^d=0$ with $G_r(u,v)$ a product of $r$ distinct linear forms, whose resolution will be detailed investigated in next section.

For later convenience, we first compute the Chern numbers and Miyaoka--Yau number of the singular surface $\ov{F}$.

\begin{ex}\label{ex: k1fbar}
The adjunction formula \eqref{eq: X-adjunction} gives
$$
K_{\ov{F}}=(K_{\bb{P}^3}+\ov{F})|_{\ov{F}}\sim(d-4)H|_{\ov{F}},
$$
where $H$ is a hyperplane section of $\bb{P}^3$. Indeed, we have $\ov{F}\sim dH$ and $K_{\bb{P}^3}\sim -4H$ (where $\sim$ denotes rational equivalence). Therefore, using Formula \eqref{eq: KX2}, we have
\begin{equation}\label{eq: c1fbar}
K_{\ov{F}}^2=d(d-4)^2.
\end{equation}

Moreover, there is a natural projection
$$
p:\ov{F}\to\bb{P}^2,\quad (x,y,z,t)\mapsto(x,y,z),
$$
which is a branched covering of degree $d$ with ramification locus $V(Q)\subseteq\bb{P}^2$, hence
$$
\chi(\ov{F})=3d-(d-1)\chi(V(Q)).
$$
The Euler characteristic number of the singular curve $V(Q)$ is 
\begin{eqnarray*}
\chi(V(Q))&=&d(3-d)+\sum_r t_r(r-1)^2,\\
\end{eqnarray*}
which implies that
\begin{eqnarray}
\chi(\ov{F})&=&d(d^2-4d+6)-(d-1)\sum_r t_r(r-1)^2\label{eq: c2fbar}
\end{eqnarray}
Consequently,
\begin{eqnarray}
MY(\ov{F})&=&3\,\chi(\ov{F})-K_{\ov{F}}^2
          =(d-1)\sum_rt_r(r-1)(3-r).\label{eq: myfbar}\qe
\end{eqnarray}
\end{ex}

\begin{rk}\label{rk: d(d-1)2}
To deduce \eqref{eq: myfbar}, we have used the following well-known equality
$$
\frac{d(d-1)}{2}=\sum_r t_r\frac{r(r-1)}{2}.
$$
\end{rk}

Let $\pi:\widetilde{F}\to \ov{F}$ be a minimal resolution of $\ov{F}$, namely, the following three conditions hold:
\begin{enumerate}[(i)]
\item $\widetilde{F}$ is a smooth surface and $\pi$ is proper birational morphism;
\item $\pi:\widetilde{F}\setminus\pi^{-1}(\Sing(\ov{F}))\to\ov{F}\setminus\Sing(\ov{F})$ is an isomorphism;
\item there is no exceptional $(-1)$-curves on $\widetilde{F}$, i.e., a rational curve $E$ on $\widetilde{F}$ such that $E^2=-1$ and $E$ is contracted to a point by $\pi$.
\end{enumerate}
Such a resolution $\pi$ can be obtained by successive embedded blowups, namely by blowing up along submanifolds of $\bb{P}^3$ as well as the resulting manifolds in each step. Note that $K_{\widetilde{F}}$ is a Cartier divisor since $\widetilde{F}$ is smooth.

\subsection{Determination of the canonical divisor}

Let $p_1,\cdots, p_s$ be all the singular points of $\ov{F}$ and $r_i$ be the multiplicity of $p_i$. Let $E_{i,1},\cdots, E_{i,v_i}$ be the irreducible components of $\pi^{-1}(p_i)$ and $M_{i,j,k}=E_{i,j}\cdot E_{i,k}$ be the intersection product of $E_{i,j}$ and $E_{i,k}$. Let moreover,
$$
\bg{M_i}=(M_{i,j,k})
$$
be the intersection matrix of $E_{i,j}$'s for any fixed $i$. It is a $v_i\times v_i$ matrix. Set
$$
\bg{E_i}=(E_{i,1},E_{i,2}\cdots, E_{i,v_i})
$$
as a $1\times v_i$ matrix.

Then the canonical divisor $K_{\widetilde{F}}$ is of the following form
$$
K_{\widetilde{F}}=\pi^*K_{\ov{F}}+\sum_{i=1}^s\sum_{j=1}^{v_i}a_{i,j}E_{i,j}.
$$
Let
$$
\bg{A_i}=(a_{i,1},a_{i,2},\cdots, a_{i,v_i})^T
$$
be a $v_i\times 1$ matrix, where $(\quad )^T$ denotes the transpose of a matrix, then $K_{\widetilde{F}}$ can be written as
\begin{eqnarray}
K_{\widetilde{F}}&=&\pi^*K_{\ov{F}}+\sum_{i=1}^s\bg{E_iA_i}.\label{eq: canA}
\end{eqnarray}
Let
$$
\bg{E_{i}}\cdot K_{\widetilde{F}}=(E_{i,1}\cdot K_{\widetilde{F}},\cdots, E_{i,v_i}\cdot K_{\widetilde{F}}),
$$
be a $1\times v_i$ matrix. By Theorem \ref{thm: resolhom} below, each $E_{i,j}$ is a smooth complete curve, and by \cite{MD1}, each $\bg{M_i}$ is a symmetric, negative definite $v_i\times v_i$ matrix. Set $\bg{N_i}=-\bg{M_i}^{-1}$.

Taking intersection product of $K_{\widetilde{F}}$ with the exceptional divisors $E_{i,j}$'s, it follows that
$$
\bg{A_i}=-\bg{N_i(E_i}\cdot K_{\widetilde{F}})^T,
$$
and hence by equality \eqref{eq: canA}, we have
\begin{equation}
\label{eq: cangood}
K_{\widetilde{F}}=\pi^*K_{\ov{F}}-\sum_{i=1}^s\bg{E_iN_i(E_i}\cdot K_{\widetilde{F}})^T.
\end{equation}
Moreover, it follows from the adjunction formula
\begin{equation}
\label{eq: ad-curve}
E_{i,j}\cdot K_{\widetilde{F}}=2g(E_{i,j})-2-E_{i,j}^2
\end{equation}
that $K_{\widetilde{F}}$ is uniquely determined once we know the genera $g(E_{i,j})$'s and the intersection matrices $\bg{M_i}$'s.

Finally, by the pull-back formula \eqref{eq: pullback} and the projection formula \eqref{eq: projection formula}, we have, by Formula \eqref{eq: cangood}, that
\begin{eqnarray*}
K_{\widetilde{F}}^2-K_{\ov{F}}^2&=&\sum_{i=1}^s\biggl(\bg{E_iN_i(E_i}\cdot K_{\widetilde{F}})^T\biggr)^2.
\end{eqnarray*}
Observe that the term $\bg{E_iN_i(E_i}\cdot K_{\widetilde{F}})^T$ involves only the resolution of the point $p_i$, hence the above formula motivates us to study the resolution of only one singularity of a normal surface or more specifically, resolution of a normal surface germ.

\subsection{Miyaoka--Yau number}
 Recall that $p_i$ is a singular point of $\ov{F}$ of multiplicity $r_i$. By definition, we have, with the above notations,
$$
DCI_{r_i,d}=\biggl(\bg{E_iN_i(E_i}\cdot K_{\widetilde{F}})^T\biggr)^2,
$$
and hence
$$
K_{\widetilde{F}}^2-K_{\ov{F}}^2=\sum_{i=1}^s DCI_{r_i,d}=\sum_r t_r DCI_{r,d}.
$$
Similarly, for the Euler number, we have
$$
\chi(\widetilde{F})-\chi(\ov{F})=\sum_{i=1}^s DCII_{r_i,d}=\sum_rt_r DCII_{r,d},
$$
therefore, it follows from \eqref{eq: myfbar} that
\begin{eqnarray}
MY(\widetilde{F})&=&MY(\ov{F})+\sum t_r DMY_{r,d}\nonumber\\
                 &=&\sum_rt_r((d-1)(r-1)(3-r)+DMY_{r,d}).\label{eq: MY}
\end{eqnarray}
For later convenience, we set
$$
\mathcal{E}_{r,d}=(d-1)(r-1)(3-r)+DMY_{r,d}.
$$

\section{Resolution of singularities}

We consider singularities of the type $(X,0): f(u,v,t)=0$ with $f(u,v,t)=G_r(u,v)+t^d$, where $G_r(u,v)$ is a product of $r$ distinct linear forms in $u,v$. Such a type of singularity in fact belongs to a special class of singularities, namely weighted homogeneous singularities, whose resolutions are explicitly known.

\subsection{Weighted homogenous singularities}

Consider the $\bb{C}^*$ action on $\bb{C}^3$ given by
$$
a\cdot(z_1,z_2,z_3)=(a^{w_1}z_1,a^{w_2}z_2,a^{w_3}z_3),
$$
where the weights $w_i=\text{weight}(z_i)$ are strictly positive integers satisfying
$$
\gcd(w_1,w_2,w_3)=1.
 $$
An isolated surface singularity $(X',0): f'(z_1,z_2,z_3)=0$ is called \emph{weighted homogeneous} of degree $N$ for the weights $w_i$ if
$$
a\cdot f'(z_1,z_2,z_3)=f'(a^{w_1}z_1,a^{w_2}z_2,a^{w_3}z_3)=a^N f'(z_1,z_2,z_3),\qquad\forall a\in\bb{C}^*.
$$

\begin{thm}[see \cite{OW} and also \cite{D92}, Section 4.10]\label{thm: resolhom}
Let $(X,0): f(u,v,t)=G_r(u,v)+t^d=0, r\leq d$ be an isolated weighted homogeneous singularity of degree $N=rd/\gcd(r,d)$, where $G_r(u,v)$ is a product of $r$ distinct linear forms in $u,v$, for the weights
$$
\begin{cases}
w_1&=\text{\rm weight}(u)=d/\gcd(r,d),\\
w_2&=\text{\rm weight}(v)=d/\gcd(r,d),\\
w_3&=\text{\rm weight}(t)=r/\gcd(r,d).
\end{cases}
$$
Then there is a resolution $\pi:\widetilde{X}\to X$ such that:
\begin{enumerate}[(i)]
\item there is a $\bb{C}^*$ action on $\widetilde{X}$ under which the morphism $\pi$ is equivariant.
\item the exceptional divisor $\pi^{-1}(0)$ has exactly one component, denoted by $E_0$, which is fixed pointwise by the $\bb{C}^*$ action on $\widetilde{X}$.
\item $\pi^{-1}(0)$ has the following form
$$
\pi^{-1}(0)=E_0\cup E_1\cup\cdots\cup E_\lambda,
$$
where for $k=1,\cdots,\lambda$,
$$
E_k=E_k^1\cup\cdots\cup E_k^r
$$
is a disjoint union of $r$ curves, corresponding to vertices at distance $k$ from the center in the dual graph below.
\item For each $k=1,\cdots,\lambda$ and $j=1,\cdots,r$, the curve $E_k^j$ is a smooth rational irreducible curve and has self-intersection $(E_k^j)^2=-n_k\leq -2$ (independent of $j$).
\item $E_0$ is a smooth complete curve of genus
\begin{eqnarray}
g(E_0)&=&\frac{1}{2}\biggl[\frac{N^2}{w_1w_2w_3}-\sum_{i<j}\frac{N\gcd(w_i,w_j)}{w_iw_j}+\sum_i\frac{\gcd(N,w_i)}{w_i}-1\biggr]\nonumber\\
    &=&\frac{1}{2}(r-2)(\gcd(r,d)-1).\label{eq: genusEc}
\end{eqnarray}
\item The components $E_0, E_k^j$'s meet transversally according to the following star-shaped graph
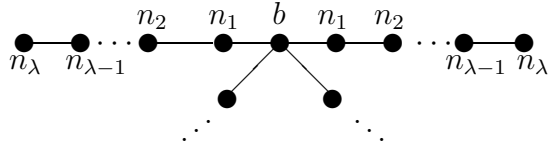
\begin{center}{\setlength{\unitlength}{0.5mm}
\begin{picture}(140,55)
\put(70,40){\circle*{5}}
\put(70,40){\line(1,0){15}}
\put(70,40){\line(-1,0){15}}
\put(70,40){\line(1,-1){13}}
\put(70,40){\line(-1,-1){13}}

\put(84,25){\circle*{5}}
\put(89,13){$\ddots$}
\put(56,25){\circle*{5}}
\put(43,13){$\iddots$}

\put(85,40){\circle*{5}}
\put(87,40){\line(1,0){15}}
\put(55,40){\circle*{5}}
\put(52,40){\line(-1,0){15}}

\put(100,40){\circle*{5}}
\put(106,38){$\cdots$}
\put(35,40){\circle*{5}}
\put(21,38){$\cdots$}

\put(119,40){\circle*{5}}
\put(121,40){\line(1,0){15}}
\put(135,40){\circle*{5}}
\put(17,40){\circle*{5}}
\put(15,40){\line(-1,0){15}}
\put(2,40){\circle*{5}}

\put(68,45){$b$}

\put(80,45){$n_1$}
\put(51,45){$n_1$}

\put(95,45){$n_2$}
\put(32,45){$n_2$}

\put(114,33){$n_{\lambda-1}$}
\put(133,33){$n_\lambda$}
\put(13,33){$n_{\lambda-1}$}
\put(-2,33){$n_\lambda$}

\end{picture}
}
\end{center}
where the central vertex corresponds to $E_0$ and there are exactly $r$ arms, which have the same length $\lambda$ and the same weight sequences $n_1,\cdots, n_\lambda$.
\item Moreover, the above dual graph satisfies the following: if we index the arms $1,2,\cdots, r$ from leftmost to right by the anticlockwise order and go along the arm indexed by $j$ from the end closest to $E_0$ to the one farthest to $E_0$, we get, in order, the vertices corresponding to the curves $E_1^j, E_2^j,\cdots, E_\lambda^j$.
\item Let $\alpha=w_1=d/\gcd(r,d)$ and $b'=w_3=r/\gcd(r,d)$. When $\alpha=1$, then there are in fact no arms, i.e., $\lambda=0$ and in this case, let $\beta=0$. When $\alpha>1$, choose $0<\beta<\alpha$ such that $\beta b' \equiv -1\Mod\alpha$. Then the weights of the vertices of the dual graph are determined as follows:
    \begin{itemize}
    \item The weight of the central vertex is
    $$
    b=\frac{N}{w_1w_2w_3}+r\beta/\alpha=\frac{\gcd(r,d)(1+b'\beta)}{\alpha}.
    $$
    \item The weight sequence $(n_1,\cdots, n_\lambda)$ along each arm is given by the following continued fraction decomposition
    $$
    \frac{\alpha}{\beta}=n_1-\frac{1}{n_2-\frac{1}{\cdots-\frac{1}{n_\lambda}}}.
    $$
    \end{itemize}
\end{enumerate}
\end{thm}

\subsection{Numerical invariants}
Let $(X,0): f(u,v,t)=G_r(u,v)+t^d=0$ be a surface germ in $(\bb{C}^3,0)$, where $G_r$ is a product of $r$ distinct linear binary forms. Let $\pi:\widetilde{X}\to X$ be the resolution given in Theorem \ref{thm: resolhom}.
With the notations in the theorem, we shall write the divisors on $\widetilde{X}$,
$$
E_k=E_k^1+E_k^2\cdots+E_k^r,\qquad k=1,\cdots, \lambda.
$$
Clearly, each $E_k$ is a Cartier divisor with compact support on $\widetilde{X}$, and
$$
E_k^2=-rn_k,\quad k=1,\cdots,\lambda,
$$
and
$$
E_k\cdot E_{k'}=
\begin{cases}
r,&\qquad\text{if }k'=k\pm1,\\
0,&\qquad\text{otherwise.}
\end{cases}
$$
Also, we can see that
$$
\chi(\widetilde{X})-\chi(X)=-1+\chi(E_0)+r\lambda .
$$
Indeed, essentially $\tilde{X}$ is obtained from $X$ by replacing 0 by $(1+r\lambda )$ curves intersecting according to the dual graph; $E_0$ contributes to $\chi(E_0)$ for $\chi(\widetilde{X})$; each arm in the dual graph gives rise to a disjoint union of $\lambda$ copies of $\bb{P}^1\setminus\{\text{\rm one point}\}\cong\bb{C}$, and hence contributes $\lambda$ for $\chi(\widetilde{X})$.

Moreover, $K_{\widetilde{X}}$ has the following form
$$
K_{\widetilde{X}}=\pi^*K_X+a_0 E_0+\sum_{k,j}a_k^jE_k^j.
$$
Set
$$
\bg{E^j}=(E_1^j,\cdots, E_\lambda^j),
\qquad
\bg{a^j}=(a_1^j,\cdots, a_\lambda^j),
$$
then
$$
K_{\widetilde{X}}=\pi^*K_X+a_0E_0+\sum_j \bg{E^j(a^j)}^T.
$$
By considering the adjunction formula, we have
$$
E_0\cdot K_{\widetilde{X}}=g(E_0)-2-E_0^2=2g(E_0)-2+b
$$
and for all $k,j$,
$$
E_k^j\cdot K_{\widetilde{X}}=g(E_k^j)-2-(E_k^j)^2=-2+n_k,
$$
hence, by the projection formula and Theorem \ref{thm: resolhom}, we get a systems of equations
\begin{equation}\label{eq: sysa}
\begin{cases}
-b a_0+(a_1^1+\cdots+a_1^r)&=(r-2)(\gcd(r,d)-1)-2+b\\
-n_k a_k^j+(a_{k-1}^j+a_{k+1}^j)&=-2+n_k,\quad \forall k,j
\end{cases}
\end{equation}
where we have denoted $a_0^j=a_0$ and $a_{\lambda+1}^j=0$ for all $j$.

The intersection matrix of $E_0, E_k^l$'s is negative definite (see \cite{MD1}), so from \eqref{eq: sysa} we can uniquely solve $a_0,\bg{a^j}$'s. Moreover, we can see that if $(a_0,\bg{a^1},\cdots, \bg{a^r})$ is a solution of the system \eqref{eq: sysa}, $(a_0,\bg{a^j},\bg{a^2},\cdots, \bg{a^{j-1}},\bg{a^j},\bg{a^{j+1}},\cdots,\bg{a^r})$ is also a solution for any $j>1$, hence from uniqueness of the solution, it follows that
$$
a_k^1=a_k^2=\cdots=a_k^r
$$
for all $k$, hence
$$
K_{\widetilde{X}}=\pi^*K_X+a_0E_0+\sum_{k=1}^\lambda  a_kE_k=\pi^*K_X+\sum_{k=0}^\lambda a_kE_k,
$$
satisfying (following from \eqref{eq: sysa})
\begin{equation}
\label{eq: sysb}
\begin{cases}
-ba_0+ra_1&=(r-2)(\gcd(r,d)-1)-2+b\\
-n_k a_k+(a_{k-1}+a_{k+1})&=-2+n_k,\quad  k=1,\cdots,\lambda,\\
\end{cases}
\end{equation}
where $a_{\lambda+1}=0$.

\subsection{Examples of resolutions}

Now we apply the notations above and consider the resolution given in Theorem \ref{thm: resolhom} of the surface germ $(X,0): G_r(u,v)+t^d=0$.

\begin{ex}\label{ex: resolA}
When $r=2$, then $(X,0)$ is a singularity of type $A_{d-1}$, and its minimal resolution $\pi:\widetilde{X}\to X$ is well-known: the dual graph is
 \begin{center}{\setlength{\unitlength}{0.5mm}
\begin{picture}(100,20)
\put(10,15){\circle*{5}}
\put(12,15){\line(1,0){25}}
\put(39,15){\circle*{5}}
\put(41,15){\line(1,0){5}}
\put(60,15){\line(1,0){5}}
\put(65,15){\circle*{5}}
\put(67,15){\line(1,0){25}}
\put(90,15){\circle*{5}}

\put(49,13){$\cdots$}
\end{picture}
}
\end{center}
where there are $(d-1)$ vertices and each vertex has weight 2. Moreover, $K_{\widetilde{X}}=\pi^*K_X$ (see \cite{Re}), so we have
$DCI_{r,d}=0, DCII_{r,d}=d-1,$
hence
$$
DMY_{r,d}=3(d-1),\qquad
\mathcal{E}_{r,d}=(d-1)(r-1)(3-r)+DMY_{r,d}=4(d-1).\qe
$$
\end{ex}

Note that when $r=2$ and $d=rp+1$ for $p\geq1$, the resolution given in Theorem \ref{thm: resolhom} is not minimal. Indeed, the central curve $E_0$ is a $(-1)$-curve, i.e. $g(E_0)=0$ and $b=1$ in Theorem \ref{thm: resolhom}. Indeed, we have the following.

\begin{prop}\label{prop: not minimal}
The resolution given in Theorem \ref{thm: resolhom} is {\bf not} minimal if and only if $d\equiv1\Mod r$.
\end{prop}

\begin{proof}
The resolution is not minimal only if $E_0$ is a $(-1)$-curve, since other exceptional irreducible curves all have self-intersection $\leq -2$. This is the case if and only if that $g(E_0)=0$ and $b=1$, namely,
$$
\begin{cases}
0&=g(E_0)=\frac{1}{2}(r-2)(\gcd(r,d)-1),\\
1&=b=\gcd(r,d)(b'\beta+1)/\alpha.
\end{cases}
$$
From the second equality, it follows that $\gcd(r,d)=1$ and $b'\beta+1=\alpha$. Now from $\gcd(r,d)=1$, we have by definition $\alpha=d/\gcd(r,d)=d$ and $b'=r/\gcd(r,d)=r$, so $d=r\beta+1$.\qedhere
\end{proof}

Consequently, if $d$ cannot be written  as $d=rp+1$ for some $p\geq1$, the resolution given in Theorem \ref{thm: resolhom} is already minimal. If, on the other hand, $d=rp+1$ for some $p\geq1$, the resolution given in Theorem \ref{thm: resolhom} is not minimal and $E_0$ is a $(-1)$-curve. By blowing down $E_0$, we get another resolution $\widetilde{X}'$ of $X$, and moreover, since in this case $\alpha=b'\beta+1=r\beta+1$, by performing the continued fraction decomposition of $\alpha/\beta=(r\beta+1)/\beta$, we have $n_1=r+1\geq 3$, hence $\widetilde{X}'$ is a minimal resolution of $X$.  But now the dual graph is
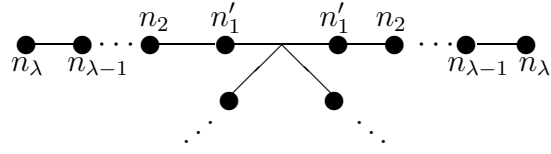
\begin{center}{\setlength{\unitlength}{0.5mm}
\begin{picture}(140,55)
\put(70,40){\line(1,0){15}}
\put(70,40){\line(-1,0){15}}
\put(70,40){\line(1,-1){13}}
\put(70,40){\line(-1,-1){13}}

\put(84,25){\circle*{5}}
\put(89,13){$\ddots$}
\put(56,25){\circle*{5}}
\put(43,13){$\iddots$}

\put(85,40){\circle*{5}}
\put(87,40){\line(1,0){15}}
\put(55,40){\circle*{5}}
\put(52,40){\line(-1,0){15}}

\put(100,40){\circle*{5}}
\put(106,38){$\cdots$}
\put(35,40){\circle*{5}}
\put(21,38){$\cdots$}

\put(119,40){\circle*{5}}
\put(122,40){\line(1,0){15}}
\put(135,40){\circle*{5}}
\put(17,40){\circle*{5}}
\put(15,40){\line(-1,0){15}}
\put(2,40){\circle*{5}}


\put(80,45){$n_1'$}
\put(51,45){$n_1'$}

\put(95,45){$n_2$}
\put(32,45){$n_2$}

\put(114,33){$n_{\lambda-1}$}
\put(133,33){$n_\lambda$}
\put(13,33){$n_{\lambda-1}$}
\put(-2,33){$n_\lambda$}

\end{picture}
}
\end{center}
where $n_1'=n_1-1$ and there is no central vertex, meaning that for the $r$ exceptional curves $E_1^1,\cdots, E_1^r$ corresponding to the vertices of weight $n_1'$, we have $E_1^j\cdot E_1^{j'}=1$ for $j\neq j'$. In particular, the new exceptional divisor does not have normal crossings.

In the sequel, by abuse of notation, we will not distinguish $\widetilde{X}$ and $\widetilde{X}'$ and always denote $\widetilde{X}$ the minimal resolution of $X$ obtained, by blowing down the central curve $E_0$ if necessary, from the resolution given in Theorem \ref{thm: resolhom}.

For later convenience, we consider specifically the case $d\equiv1\Mod r$.

\begin{ex}\label{ex: resolC}
Let $r\geq 3$ and $d=rp+1, p\geq1$. Then by Proposition \ref{prop: not minimal}, we have $g(E_0)=0$ and $b=1$.
According to Theorem \ref{thm: resolhom}, we have  $\alpha=d/\gcd(r,d)=d$ and $b'=r/\gcd(r,d)=r$, so $\alpha=b'p+1;$
since $0<\beta<\alpha$ is chosen so that $b'\beta\equiv-1\Mod\alpha$, we have $\beta=p$. In addition,
$$
\alpha/\beta=(rp+1)/p,
$$
so considering the continued fraction decomposition, we have
$$
\lambda=p,\qquad
n_1=r+1,\qquad n_2=n_3=\cdots=n_\lambda=2.
$$

Blowing $E_0$ down, we get the minimal resolution $\pi:\widetilde{X}\to X$.
The canonical divisor $K_{\widetilde{X}}$ has the following form
$$
K_{\widetilde{X}}=\pi^*K_X+\sum_{k=1}^\lambda a_kE_k,
$$
where $E_1=E_1^1+\cdots+E_1^r$ such that $(E_1^l)^2=-n_1'=-(n_1-1)$ and $E_1^l\cdot E_1^{l'}=1$ for $l<l'$.

Now taking the intersection product of $K_{\widetilde{X}}$ with $E_k^l$'s and applying the adjunction formula and projection formula, we have
$$
\begin{cases}
-n_1'a_1+a_2+(r-1)a_1&=-2+n_1'\\
-n_2a_2+(a_1+a_3)&=-2+n_2\\
-n_3a_2+(a_2+a_4)&=-2+n_3\\
\qquad\vdots &\\
-n_{\lambda-1}a_{\lambda-1}+(a_{\lambda-2}+a_\lambda)&=-2+n_{\lambda-1}\\
-n_\lambda a_\lambda+a_{\lambda-1}&=-2+n_\lambda
\end{cases}
$$
that is,
$$
\begin{cases}
-ra_1+a_2+(r-1)a_1&=-2+r\\
-2a_2+(a_1+a_3)&=0\\
-2a_2+(a_2+a_4)&=0\\
\qquad\vdots &\\
-2a_{\lambda-1}+(a_{\lambda-2}+a_\lambda)&=0\\
-2a_\lambda+a_{\lambda-1}&=0.
\end{cases}
$$
By considering from bottom equation to the second top one, we have
$$
a_k=(\lambda+1-k)a_\lambda,\quad k=1,\cdots,\lambda-1;
$$
hence from the first equation, we get $a_\lambda=-(r-2)$. It follows that
$$
K_{\widetilde{X}}=\pi^*K_X-(r-2)(E_p+2E_{p-1}+\cdots+pE_1),
$$
and thus,
$$
DCI_{r,d}=(r-2)^2(E_p+2E_{p-1}+\cdots+pE_1)^2.
$$
Note that $E_k^2=-rn_k=-2r$ for $k>1$ and
\begin{eqnarray*}
E_1^2&=&(E_1^1+\cdots+E_1^r)^2
     =-r;
\end{eqnarray*}
moreover, $E_k\cdot E_{k'}=r$ for $k'=k\pm 1$, =0 otherwise. Hence, we have
$$
DCI_{r,d}=(r-2)^2(E_p+2E_{p-1}+\cdots+pE_1)^2
         =-(r-2)^2rp
         =-(d-1)(r-2)^2.
$$
In addition, we have
$$
DCII_{r,d}=r\lambda =rp=d-1,
$$
thus
$$
DMY_{r,d}=3\,DCII_{r,d}-DCI_{r,d}
         =3(d-1)+(d-1)(r-2)^2.
$$
Consequently,
$$
\mathcal{E}_{r,d}=DMY_{r,d}+(d-1)(r-1)(3-r)
                 = 4(d-1).\qe
$$
\end{ex}

\section{Numerical invariants for minimal resolutions}\label{sec: general invariants}

Now we consider the general case of Theorem \ref{thm: resolhom}. Although our method applies for more general situations, we assume $r\geq3$ and $d\not\equiv1\Mod r$, since otherwise we are done by Example \ref{ex: resolA} and Example \ref{ex: resolC}. In particular, the resolution $\pi:\widetilde{X}\to X$ given in Theorem \ref{thm: resolhom} is a minimal resolution.

\subsection{Continued fraction decomposition}

In order to apply Theorem \ref{thm: resolhom}, we first deal with the continued fraction decomposition
$$
\frac{\alpha}{\beta}=n_1-\frac{1}{n_2-\frac{1}{\cdots-\frac{1}{n_\lambda}}}.
$$
Recall that $\beta$ is chosen such that $b'\beta\equiv-1\Mod\alpha$, hence $\gcd(\alpha,\beta)=1$. Let $$
\alpha_0,\alpha_1,\cdots,\alpha_{\lambda-1},\alpha_\lambda=1,\alpha_{\lambda+1}=0
$$ be a sequence of natural numbers such that $\gcd(\alpha_i,\alpha_{i+1})=1$ for $i=0,1,\cdots, \lambda$ and
\begin{equation}\label{eq: conalphai}
\frac{\alpha_i}{\alpha_{i+1}}=n_{i+1}-\frac{1}{n_{i+2}-\frac{1}{\cdots-\frac{1}{n_\lambda}}},\qquad i=0,1,\cdots, \lambda-1.
\end{equation}
Clearly, the numbers $\alpha_i$'s are uniquely determined by the continued fraction decomposition above, and $\alpha_i>0$ for $i<\lambda+1$.

Moreover, we have by definition \eqref{eq: conalphai}
$$
\frac{\alpha_{i-1}}{\alpha_{i}}=n_i-\frac{1}{\alpha_i/\alpha_{i+1}}=\frac{n_i\alpha_i-\alpha_{i+1}}{\alpha_i},
$$
hence
$$
\alpha_{i-1}=n_i\alpha_i-\alpha_{i+1},
$$
or in another more convenient formulation
\begin{equation}\label{eq: realphai}
\begin{pmatrix}
\alpha_{i-1}\\
\alpha_i
\end{pmatrix}
=\begin{pmatrix}
n_i & -1\\
1 & 0
\end{pmatrix}
\begin{pmatrix}
\alpha_i\\
\alpha_{i+1}
\end{pmatrix}.
\end{equation}

Set for $i=1,\cdots, \lambda$,
\begin{equation}\label{eq: conGi}
\bg{G_i}=
\begin{pmatrix}
n_i & -1\\
1 & 0
\end{pmatrix}
\end{equation}
be a $2\times 2$ matrix. Then the relation \eqref{eq: realphai} can be formulated as
$$
\begin{pmatrix}
\alpha_{i-1}\\
\alpha_i
\end{pmatrix}
=\bg{G_i}
\begin{pmatrix}
\alpha_i\\
\alpha_{i+1}
\end{pmatrix}.
$$
Thus, we have
\begin{equation}\label{eq: foralphai}
\begin{pmatrix}
\alpha_{i-1}\\
\alpha_i
\end{pmatrix}
=\bg{G_iG_{i+1}\cdots G_\lambda}
\begin{pmatrix}
\alpha_\lambda\\
\alpha_{\lambda+1}
\end{pmatrix}=\bg{G_iG_{i+1}\cdots G_\lambda}
\begin{pmatrix}
1\\
0
\end{pmatrix}
\end{equation}
for all $i\geq 1$.

Note also that by definition \eqref{eq: conalphai} and our conventions, $\alpha_0=\alpha$ and $\alpha_1=\beta$.

Let
$$
\bg{G}=
\bg{G_1G_2\cdots G_\lambda},
$$
then by \eqref{eq: foralphai}, we have
$$
\begin{pmatrix}
\alpha\\
\beta
\end{pmatrix}=
\begin{pmatrix}
\alpha_0\\
\alpha_1
\end{pmatrix}=\bg{G}
\begin{pmatrix}
1\\
0
\end{pmatrix}.
$$
So $\bg{G}$ is of the form
$$
\bg{G}=\begin{pmatrix}
\alpha & \gamma\\
\beta & \delta
\end{pmatrix}
$$
for some integers $\gamma,\delta$. In fact, we have the following more precise result.

\begin{prop}\label{prop: entryG}
With the notations as above and in Theorem \ref{thm: resolhom}, we have
$$
\bg{G}=\begin{pmatrix}
\alpha & b'-\alpha\\
\beta & \frac{1+b'\beta}{\alpha}-\beta
\end{pmatrix},
$$
namely, $\gamma=b'-\alpha$ and $\delta=-\beta+(1+b'\beta)/\alpha$.
\end{prop}
\begin{proof}
First, we show
\begin{claim}\label{claim: alphabetagammadelta}
$-\alpha<\gamma\leq 0$ and $-\beta<\delta\leq 0$.
\end{claim}

Assuming the claim, note that by definition,
$$
\det\bg{G}=\alpha\delta-\beta\gamma=1,
$$
hence $\beta\gamma\equiv-1\Mod\alpha$. Recall also the $b'\beta\equiv-1\Mod\alpha$, so we have $\gamma=b'-\alpha$ since $\gamma, b'-\alpha\in(-\alpha,0]$ and the equation $\beta x\equiv-1\Mod\alpha$ admits a unique solution satisfying $x\in(-\alpha,0]$. In addition,
$$
\delta=\frac{1+\beta\gamma}{\alpha}=\frac{1+\beta(b'-\alpha)}{\alpha}=\frac{1+b'\beta}{\alpha}-\beta.
$$

{\it Proof of Claim \ref{claim: alphabetagammadelta}: } For $i\geq 1$, let
$$
\begin{pmatrix}
\xi_i & \gamma_i\\
\eta_i & \delta_i
\end{pmatrix}
=\bg{G_1G_2\cdots G_i},
$$
then $\xi_i,\eta_i,\gamma_i,\delta_i$ are all integers. It suffices to show the following:
\begin{enumerate}[(i)]
\item $\xi_i,\eta_i>0$ for all $i$.
\item $\gamma_i\in(-\xi_i,0]$ and $\delta_i\in(-\eta_i,0]$ for all $i$.
\end{enumerate}
We prove this by induction on $i$. When $i=1$, then we have
$$
\begin{pmatrix}
\xi_1 & \gamma_1\\
\eta_1 & \delta_1
\end{pmatrix}=\bg{G_1}=\begin{pmatrix}
n_1 & -1\\
1 & 0
\end{pmatrix},
$$
and the conclusion obviously holds. Now assuming the validity of the result for $i$, we have
$$
\begin{pmatrix}
\xi_{i+1} & \gamma_{i+1}\\
\eta_{i+1} & \delta_{i+1}
\end{pmatrix}=\begin{pmatrix}
\xi_i & \gamma_i\\
\eta_i & \delta_i
\end{pmatrix}\bg{G_{i+1}}=\begin{pmatrix}
\xi_i & \gamma_i\\
\eta_i & \delta_i
\end{pmatrix}\begin{pmatrix}
n_{i+1} & -1\\
1 & 0
\end{pmatrix}.
$$
Therefore,
\begin{enumerate}[i)]
\item $\xi_{i+1}=n_{i+1}\xi_i+\gamma_i>2\xi_i-\xi_i>0$ since $n_{i+1}\geq 2$ and by inductive hypothesis, $\xi_i>0$ and $\gamma_i\in(-\xi_i,0]$. Similarly, $\eta_{i+1}=n_{i+1}\eta_i+\gamma_i>0$ by the inductive hypothesis $\eta_i>0$ and $\gamma_i\in(-\eta_i,0]$.
\item $\gamma_{i+1}=-\xi_i<0$ since $\xi_i>0$; in addition,
$$
\gamma_{i+1}+\xi_{i+1}=-\xi_i+(n_{i+1}\xi_i+\gamma_i)>(n_{i+1}-2)\xi_i\geq 0,
$$
since $n_{i+1}\geq 2$ and $\gamma_i>-\xi_i$ by the inductive hypothesis. Similarly, $\delta_{i+1}=-\eta_i<0$ and
$$
\delta_{i+1}+\eta_{i+1}=-\eta_i+(n_{i+1}\eta_i+\delta_i)>(n_{i+1}-2)\eta_i\geq 0.
$$
\end{enumerate}
We are done.\qedhere
\end{proof}

\subsection{Formulae for the canonical divisor}

As before, we assume
$$
K_{\widetilde{X}}=\pi^*K_X+\sum_{i=0}^\lambda a_iE_i.
$$
Therefore,
$$
DCI_{r,d}=\sum_{i=0}^\lambda  a_i^2E_i^2+2\sum_{i=0}^{\lambda-1}a_ia_{i+1}E_i\cdot E_{i+1}.
$$
Recall that $E_0^2=-b$ and $E_i^2=-rn_i$ for $i>0$. In addition, $E_i\cdot E_{i'}=r$ for $i'=i\pm1$, =0 otherwise. Hence, we have
$$
DCI_{r,d}=-ba_0^2+r\sum_{i=1}^\lambda  a_i(2a_{i-1}-n_ia_i).
$$
By \eqref{eq: sysb}, we have $-n_ia_i+a_{i-1}+a_{i+1}=-2+n_i$, so
\begin{eqnarray}
DCI_{r,d}&=&a_0(-ba_0+ra_1)+r\biggl(\sum_{i=1}^\lambda  n_ia_i-2\sum_{i=1}^\lambda a_i\biggr).\label{eq: dc1ni2}
\end{eqnarray}
By \eqref{eq: sysb}, we have $-n_ia_i+a_{i-1}+a_{i+1}=-2+n_i$, so
\begin{eqnarray*}
\sum_{i=1}^\lambda  n_ia_i-2\sum_{i=1}^\lambda a_i&=&\sum_{i=1}^\lambda \biggl(a_{i-1}+a_{i+1}-(n_i-2)\biggr)-2\sum_{i=1}^\lambda a_i\\
                                    &=&-\sum_{i=1}^\lambda (n_i-2)+(a_0-a_1-a_\lambda);
\end{eqnarray*}
consequently, by \eqref{eq: dc1ni2}, we get
\begin{eqnarray}
DCI_{r,d}&=&a_0(-ba_0+ra_1)-2+b)-r\sum_{i=1}^\lambda (n_i-2)+r(a_0-a_1-a_\lambda).\label{eq: dc1a0a1aq}
\end{eqnarray}

Now we compute $a_0$, $a_1$ and $a_\lambda$. By equation \eqref{eq: sysb}, we have
$$
\begin{cases}
-b a_0+ra_1&=(r-2)(\gcd(r,d)-1)-2+b\\
-n_1a_1+(a_0+a_2)&=-2+n_1\\
-n_2a_2+(a_1+a_3)&=-2+n_2\\
\qquad\vdots &\\
-n_{\lambda-1}a_{\lambda-1}+(a_{\lambda-2}+a_\lambda)&=-2+n_{\lambda-1}\\
-n_\lambda a_\lambda+a_{\lambda-1}&=-2+n_\lambda.
\end{cases}
$$
Let $a_i^*=a_i+1$ for $i=0,1,\cdots, \lambda+1$. Recall also that $a_{\lambda+1}=0$. Then the above equations can be reformulated into a more convenient form:
$$
\begin{cases}
-b a_0^*+ra_1^*&=\gcd(r,d)(r-2)\\
-n_1a_1^*+(a_0^*+a_2^*)&=0\\
-n_2a_2^*+(a_1^*+a_3^*)&=0\\
\qquad\vdots &\\
-n_{\lambda-1}a_{\lambda-1}^*+(a_{\lambda-2}^*+a_\lambda^*)&=0\\
-n_\lambda a_\lambda^*+(a_{\lambda-1}^*+a_{\lambda+1}^*)&=0.
\end{cases}
$$
With the help of the matrices $\bg{G_i}$ defined in \eqref{eq: conGi}, we have
$$
\begin{pmatrix}
a_{i-1}^*\\
a_i^*
\end{pmatrix}=\bg{G_i}\begin{pmatrix}
a_i^*\\
a_{i+1}^*
\end{pmatrix},
$$
hence
$$
\begin{pmatrix}
a_0^*\\
a_1^*
\end{pmatrix}=
\bg{G_1\cdots G_\lambda}\begin{pmatrix}
a_\lambda^*\\
a_{\lambda+1}^*
\end{pmatrix}=\bg{G}\begin{pmatrix}
a_\lambda^*\\
1
\end{pmatrix}.
$$
By Proposition \ref{prop: entryG}, we thus have
\begin{equation}\label{eq: a0a1}
\begin{cases}
a_0^*&=\alpha a_\lambda^*+(b'-\alpha)=\alpha a_\lambda+b'\\
a_1^*&=\beta a_\lambda^*+(\frac{1+b'\beta}{\alpha}-\beta)=\beta a_\lambda+(1+b'\beta)/\alpha.
\end{cases}
\end{equation}
Furthermore, it also holds $-b a_0^*+ra_1^*=\gcd(r,d)(r-2)$, thus we obtain three equations in $a_0,a_1,a_\lambda$. The solution is as follows, whose proof involves only direct computations and is left to the reader.
\begin{lem}
$$
\begin{cases}
a_0&=(2-r)\alpha+b'-1,\\
a_1&=(2-r)\beta+\frac{1+b'\beta}{\alpha}-1,\\
a_\lambda&=-(r-2).
\end{cases}
$$
\end{lem}
As a corollary, it follows from \eqref{eq: dc1a0a1aq} that
\begin{equation}\label{eq: dc1result}
DCI_{r,d}=-d(r-2)^2-r\sum_{i=1}^\lambda (n_i-2)+2(r-2)(r-\gcd(r,d))+(r-b).\\
\end{equation}
It is clear, in the above formula, that the first term $-d(r-2)^2$ is obviously negative, and also $-r\sum_{i=1}^\lambda (n_i-2)\leq 0$ since $n_i\geq 2$. Moreover, $2(r-2)(r-\gcd(r,d))\geq 0$ since $r\geq 3$ and $r\geq\gcd(r,d)$; in addition,
$$
r-b=\gcd(r,d)\biggl(b'-\frac{1+b'\beta}{\alpha}\biggr)=\frac{\gcd(r,d)}{\alpha}(b'(\alpha-\beta)-1)
$$
is nonnegative since $\alpha-\beta>0$ and $b'\geq 1$.

\subsection{Estimations of the Miyaoka-Yau numbers}

We consider the Miyaoka-Yau number of the minimal resolution $\pi: \widetilde{F}\to\ov{F}$. Then we have
\begin{equation}\label{eq: dc2result}
DCII_{r,d}=-1+\chi(E_0)+r\lambda =1+r\lambda -(r-2)(\gcd(r,d)-1).
\end{equation}
By definition, $DMY_{r,d}=3\,DCII_{r,d}-DCI_{r,d}$. Hence, in view of \eqref{eq: dc1result}, we get
\begin{eqnarray*}
DMY_{r,d}&=&\biggl(3(1+r\lambda )+r\sum_{i=1}^\lambda (n_i-2)\biggr)+(d-1)(r-2)^2\\
         & &-\biggl((r-2)(\gcd(r,d)+r-1)+(r-b)\biggr).
\end{eqnarray*}
In addition, $\mathcal{E}_{r,d}=DMY_{r,d}+(d-1)(r-1)(3-r)$ by definition, it follows that
\begin{eqnarray}
\mathcal{E}_{r,d}&=&\biggl(r\sum_{i=1}^\lambda (n_i+1)\biggr)+(d+2)\nonumber\\
                 & &-\biggl((r-2)(\gcd(r,d)+r-1)+(r-b)\biggr). \label{eq: erdresult}
\end{eqnarray}

Now we begin to estimate $\mathcal{E}_{r,d}$. First, we have
$$
 (r-2)(\gcd(r,d)+r-1)+(r-b)
 \leq(r-2)(2r-1)+r
 =2(r-1)^2,
$$
so the following hold:
\begin{eqnarray}
\mathcal{E}_{r,d}&\geq&\biggl(r\sum_{i=1}^\lambda (n_i+1)\biggr)+(d+2)-2(r-1)^2
                 >-2r(r-1).\label{eq: erdest}
\end{eqnarray}

\begin{rk}
The above estimate is also true when $d\equiv1\Mod r$ by Example \ref{ex: resolC} and when $r=2$ by Example \ref{ex: resolA}.
\end{rk}

As an application of the above calculations, we give new examples of computing Chern numbers and $\mathcal{E}_{r,d}$ by directly using Formulae \eqref{eq: dc1result}, \eqref{eq: dc2result} and \eqref{eq: erdresult}.

\begin{ex}\label{ex: resolB}
Let $r\geq 3$ and $d=rp, p\geq 1$. Then the resolution $\pi:\widetilde{X}\to X$ in Theorem \ref{thm: resolhom} is minimal. Then we have
\begin{enumerate}[(i)]
\item $\gcd(r,d)=r$, so $\alpha=d/\gcd(r,d)=p$ and $b'=r/\gcd(r,d)=1$. Since $\alpha=b'p$ and by assumption $\beta$ is chosen so that $0\leq\beta<\alpha$ satisfying $b'\beta\equiv-1\Mod\alpha$, we have $\beta=p-1$.
\item We get
$$
b=\frac{\gcd(r,d)(1+b'\beta)}{\alpha}=r.
$$
\item We have
$$
\frac{\alpha}{\beta}=\frac{p}{p-1},
$$
doing the continued fraction decomposition, we see that
$$
\lambda=p-1,\qquad
n_1=\cdots=n_\lambda=2.
$$
Therefore,
$$
r\sum_{i=1}^\lambda (n_i+1)=3r\lambda =3r(p-1)=3d-3r.
$$
\item Eventually, by \eqref{eq: dc1result}, we have
\begin{eqnarray*}
DCI_{r,d}&=&-d(r-2)^2;
\end{eqnarray*}
by \eqref{eq: dc2result}, we have
$$
DCII_{r,d}=1+r(p-1)-(r-2)(r-1)
         =d-(r-1)^2;
$$
by \eqref{eq: erdresult}, we obtain
\begin{eqnarray*}
\mathcal{E}_{r,d}&=&(3d-3r)+(d+2)-\biggl((r-2)(2r-1)+0\biggr)\\
                 &=&4d-2r(r-1).
\end{eqnarray*}
\end{enumerate}
\end{ex}

\begin{ex}\label{ex: resolD}
Let $r\geq 3$ and $d=r(p-1)+(r-1)=rp-1$ for $p\geq 2$. Then we have
\begin{enumerate}[(i)]
\item $\gcd(r,d)=1$, so $\alpha=d/\gcd(r,d)=d$ and $b'=r/\gcd(r,d)=r$. Since $\alpha=b'p-1$ and by assumption $\beta$ is chosen so that $0\leq\beta<\alpha$ satisfying $b'\beta\equiv-1\Mod\alpha$, we have $\beta=\alpha-p=p(r-1)-1$.
\item We get
$$
b=\frac{\gcd(r,d)(1+b'\beta)}{\alpha}=r-1.
$$
\item We have
$$
\frac{\alpha}{\beta}=\frac{d}{r(p-1)-1}=\frac{rp-1}{r(p-1)-1},
$$
doing the continued fraction decomposition, we see that
$
\lambda=p+r-3,
$
and
$$
n_1=\cdots=n_{p-2}=2, \qquad n_{p-1}=3,\qquad n_{p}=n_{p+1}=\cdots=n_{p+r-3}=2.
$$
Therefore,
$$
r\sum_{i=1}^\lambda (n_i+1)=r(3\lambda+1)=3r(p+r)-8r=3(d+1)+3r^2-8r.
$$
\item Eventually,  by \eqref{eq: dc1result}, we have
\begin{eqnarray*}
DCI_{r,d}&=&-d(r-2)^2+(2r-5)(r-1);
\end{eqnarray*}
by \eqref{eq: dc2result}, we have
$$
DCII_{r,d}=1+r(p+r-3)
         =d+(r-1)(r-2);
$$
by \eqref{eq: erdresult}, we obtain
\begin{eqnarray*}
\mathcal{E}_{r,d}&=&\biggl(3(d+1)+3r^2-8r\biggr)+(d+2)-\biggl(r(r-2)+1\biggr)\\
                 &=&4(d+1)+2r(r-3).
\end{eqnarray*}
\end{enumerate}
Consequently, when $r=3$, we have the following:
\begin{enumerate}[(1)]
\item when $3|d$, we have $\mathcal{E}_{3,d}=4d-12$ by Example \ref{ex: resolB};
\item when $d\equiv1\Mod3$, we have $\mathcal{E}_{3,d}=4(d-1)$ by Example \ref{ex: resolC};
\item when $d\equiv2\Mod 3$, we have $\mathcal{E}_{3,d}=4(d+1)$ by the results above.
\end{enumerate}
In particular, when $d\geq 4$, it is always true that $\mathcal{E}_{3,d}\geq 4(d-3)$.
\end{ex}

\section{Proof of Theorem \ref{thm: not ball quotient}}

Let $\pi:\widetilde{F}\to\ov{F}$ be the minimal resolution obtained in previous sections. We prove that $MY(\widetilde{F})\neq 0$ under the assumption of Theorem \ref{thm: not ball quotient}.

The proof will be divided into three cases with respect to the values of $t_d$ and $t_{d-1}$.

\begin{enumerate}[i)]
\item
When the lines in $\mathcal{A}$ form a pencil, namely, $t_d=1$, we have, $d\neq 3$ by the assumption of Theorem \ref{thm: not ball quotient}; moreover,  by Example \ref{ex: resolA}, Example \ref{ex: resolB} and Formula \eqref{eq: MY},
$$
MY(\widetilde{F})=\mathcal{E}_{d,d}=4d-2d(d-1)=2d(3-d).
$$

\item If $t_d=0$ while $t_{d-1}\neq 0$, then we have $t_{d-1}=1$ and $t_2=d-1$ (if $d=3$, $t_2=d=3$). Moreover, by Example \ref{ex: resolA} and Example \ref{ex: resolC}, in view of \eqref{eq: MY}, we have
\begin{eqnarray*}
MY(\widetilde{F})&=&t_2\mathcal{E}_{2,d}+t_{d-1}\mathcal{E}_{r,d}\\
                 &=&(d-1)\biggl(4(d-1)\biggr)+4(d-1)\\
                 &=&4d(d-1)>0.
\end{eqnarray*}
\item Now we consider the case $t_d=0, t_{d-1}=0$. Then by the estimation \eqref{eq: erdest}, we have
\begin{eqnarray*}
MY(\widetilde{F})&=&t_2\mathcal{E}_{2,d}+t_3\mathcal{E}_{3,d}+\sum_{r\geq4}t_r\mathcal{E}_{r,d}\\
                 &\geq&t_2\mathcal{E}_{2,d}+t_3\mathcal{E}_{3,d}-2\sum_{r\geq 4}t_rr(r-1)\\
                 &=&\biggl(t_2(\mathcal{E}_{2,d}+4)+t_3(\mathcal{E}_{3,d}+12)\biggr)-2\sum_rt_rr(r-1).
\end{eqnarray*}
From Remark \ref{rk: d(d-1)2}, we have $\sum_rt_rr(r-1)=d(d-1)$; moreover, from Example \ref{ex: resolA} and the end of Example \ref{ex: resolD}, we deduce that
\begin{eqnarray}
MY(\widetilde{F})&\geq&4d(t_2+t_3)-2d(d-1)=2d\biggl(2(t_2+t_3)-(d-1)\biggr).\label{eq: MYFtilde}
\end{eqnarray}
Now we use the celebrated inequality in the second remark added in proof of \cite{HI}, which states that
$$
t_2+\frac{3}{4}t_3\geq d+\sum_{r\geq 5}(r-4)t_r,
$$
see also \cite{Sa} or Appendix A of \cite{Tr}. In particular, $t_2+t_3\geq d$. It follows immediately, by \eqref{eq: MYFtilde}, that
$$
MY(\widetilde{F})>0.
$$
\end{enumerate}
The proof now is complete.\qed

\begin{rk}\label{rk: exceptional case}
When $d=|\mathcal{A}|=3$ and $\mathcal{A}$ is a pencil, i.e., $t_3=1$, then $MY(\widetilde{F})=0$. Moreover, from Example \ref{ex: resolB}, we obtain $DCI_{3,3}=-3$. Hence, by Example \ref{ex: k1fbar}, we have
$$
c_1^2(\widetilde{F})=K_{\widetilde{F}}^2=K_{\ov{F}}^2+DCI_{3,3}=3\times(3-4)^2-3=0.
$$
Moreover, $c_2(\widetilde{F})=0$ since $MY(\widetilde{F})=3c_2(\widetilde{F})-c_1^2(\widetilde{F})=0$.

Since $c_2>0$ for a smooth projective surface of general type (see \cite{BHPV}, Chapter VII), it follows that $\widetilde{F}$ is not of general type. In particular, $\widetilde{F}$ is not a ball quotient.
\end{rk}

\section{Surfaces of general type associated to line arrangements}\label{sec: general type}

Let $\mathcal{A}$ be a line arrangement in $\bb{P}^2$. By Theorem \ref{thm: not ball quotient}, $MY(\widetilde{F})>0$ when $\mathcal{A}$ is not a pencil; it is natural to ask whether $\widetilde{F}$ is a surface of general type. Inspired by \cite{HI}, it is natural to conjecture that $\widetilde{F}$ is of general type if $\mathcal{A}$ is not too singular, i.e., $t_r=0$ for $r$ large compared with $d$.

\subsection{A general type criterion}
We first provide a criterion for a surface to be of general type.

\begin{prop}\label{prop: criterion general1}
Let $X$ be a smooth projective surface. If $c_1^2(X)>9$, then $X$ is of general type.
\end{prop}
\begin{proof}
Let $X'$ be a minimal model of $X$. Then $X'$ is obtained by successively blowing down $(-1)$-curves. Note that once we blow down a $(-1)$-curve, $c_1^2$ increases by 1, so $c_1^2(X')\geq c_1^2(X)>9$, hence by the Enriques-Kodaira classification of surfaces (see \cite{BHPV}, Chapter VI), $X'$ is of general type, and thus so is $X$.\qedhere
\end{proof}

\subsection{Surfaces associated to line arrangements with only nodes and triple points}
In the sequel, we consider surfaces associated to line arrangements such that $t_r=0$ whenever $r\geq 4$, and we prove Theorem \ref{thm: general type}.

For $r=2$, by Example \ref{ex: resolA}, we have $DCI_{2,d}=0$ and $DCII_{2,d}=d-1$.

When $r=3$, we have
\begin{enumerate}[(i)]
\item If $3|d$, $DCI_{3,d}=-d, DCII_{3,d}=d-4$ by Example \ref{ex: resolB};
\item If $d\equiv1\Mod 3$, $DCI_{3,d}=-(d-1), DCII_{3,d}=d-1$ by Example \ref{ex: resolC};
\item If $d\equiv2\Mod 3$, we have $DCI_{3,d}=-(d-2), DCII_{r,d}=d+2$ by Example \ref{ex: resolD}
\end{enumerate}

\subsubsection{Case $3|d$}
When $d=3p$, we have by \eqref{eq: c1fbar} that $$
    c_1^2(\widetilde{F})= K_{\ov{F}}^2+\sum_rt_r DCI_{r,d}
                        =d(d-4)^2-d t_3
                        =d((d-4)^2-t_3),
    $$
    and by \eqref{eq: c2fbar},
    $$
    c_2(\widetilde{F})=\chi(\ov{F})+\sum_rt_r DCII_{r,d}=d(d^2-4d+6)-3t_3d.
    $$

 By Remark \ref{rk: d(d-1)2}, we have $2t_2+6t_3=d(d-1)$, hence $t_3\leq d(d-1)/6$ and thus when $d=3p\geq 9$,
 $$
 c_1^2(\widetilde{F})\geq d\biggl((d-4)^2-\frac{d(d-1)}{6}\biggr)>9.
 $$
Therefore $\widetilde{F}$ is of general type by Proposition \ref{prop: criterion general1}. In addition,
$$
\frac{c_1^2(\widetilde{F})}{c_2(\widetilde{F})}=\frac{(d-4)^2-t_3}{d^2-4d+6-3t_3}=\frac{1}{3}\biggl(1+\frac{2(d-3)(d-7)}{d^2-4d+6-3t_3}\biggr).
$$

\subsubsection{Case $d\equiv1\Mod 3$}
When $d=3p+1$, we have
    $$
    c_1^2(\widetilde{F})= K_{\ov{F}}^2+\sum_rt_r DCI_{r,d}
                        =d(d-4)^2-(d-1)t_3,
    $$
    and
    $$
    c_2(\widetilde{F})=\chi(\ov{F})+\sum_rt_r DCII_{r,d}=d(d^2-4d+6)-3(d-1)t_3.
    $$
Since $2t_2+6t_3=d(d-1)$, we have $t_3\leq d(d-1)/6$ so, when $p\geq 2$ or equivalently $d\geq7$,
$$
c_1^2(\widetilde{F})\geq d(d-4)^2-\frac{1}{6}d(d-1)^2>9,
$$
hence, $\widetilde{F}$ is of general type by Proposition \ref{prop: criterion general1}. In addition,
$$
\frac{c_1^2(\widetilde{F})}{c_2(\widetilde{F})}=\frac{d(d-4)^2-(d-1)t_3}{d(d^2-4d+6)-3(d-1)t_3}=\frac{1}{3}\biggl(1+\frac{2d(d-3)(d-7)}{d(d^2-4d+6)-3(d-1)t_3}\biggr).
$$

\subsubsection{Case $d\equiv2\Mod 3$}

When $d=3p+2$, we have
    $$
    c_1^2(\widetilde{F})= K_{\ov{F}}^2+\sum_rt_r DCI_{r,d}
                        =d(d-4)^2-(d-2)t_3,
   $$
   and
   $$
    c_2(\widetilde{F})=\chi(\ov{F})+\sum_rt_r DCII_{r,d}=d(d^2-4d+6)-3(d-2)t_3.
    $$

Since $2t_2+6t_3=d(d-1)$, we have $t_3\leq d(d-1)/6$ so, when $p\geq 2$ or equivalently $d\geq8$,
\begin{eqnarray*}
c_1^2(\widetilde{F})&\geq& d(d-4)^2-\frac{1}{6}d(d-1)(d-2)>9,
\end{eqnarray*}
hence, $\widetilde{F}$ is of general type by Proposition \ref{prop: criterion general1}. In addition,
$$
\frac{c_1^2(\widetilde{F})}{c_2(\widetilde{F})}=\frac{d(d-4)^2-(d-2)t_3}{d(d^2-4d+6)-3(d-2)t_3}=\frac{1}{3}\biggl(1+\frac{2d(d-3)(d-7)}{d(d^2-4d+6)-3(d-2)t_3}\biggr).
$$

Conclusion: in any case, $\frac{c_1^2(\widetilde{F})}{c_2(\widetilde{F})}$ is an increasing function in $t_3$ with fixed $d\geq 7$. As $t_3\leq d(d-1)/6$, it follows that
$$
1\leq\liminf_{d\to\infty}\frac{c_1^2(\widetilde{F})}{c_2(\widetilde{F})}\leq \limsup_{d\to\infty}\frac{c_1^2(\widetilde{F})}{c_2(\widetilde{F})}\leq\frac{5}{3}.
$$

Moreover, Theorem \ref{thm: general type} follows from the above discussions.

\section{Chern numbers and Hodge numbers}\label{sec: chern and hodge}
In \cite{HI}, some line arrangements are given so that they give ball quotients through the construction via Kummer covers; such arrangements includes the Hesse arrangement and the arrangement $\mathcal{A}(2,2,3): (x^2-y^2)(y^2-z^2)(z^2-x^2)=0$. By Theorem \ref{thm: not ball quotient}, these arrangements do not give ball quotient through our approach. Instead, we compute the Hodge numbers of the associated surfaces.

\subsection{Relations between Hodge numbers and Chern numbers}
In this section, we shall fix a smooth projective surface $X$. Denote $q=h^{0,1}(X)$ its irregularity and $p=h^{0,2}(X)$ its geometric genus. Denote also $b_i, i=1,2,3,4$ the Betti numbers of $X$ and $c_1^2,c_2$ the Chern numbers, as well as $h^{p,q}$ the Hodge numbers.

Then by Noether's formula (see \cite{BHPV}), we first have
\begin{equation}\label{eq: noether}
1-q+p=\frac{1}{12}(c_1^2+c_2);
\end{equation}
secondly, from the formula for Euler characteristic, we have
\begin{equation}\label{eq: euler}
2-2b_1+b_2=c_2.
\end{equation}
Moreover, from Hodge decomposition and Serre duality, we have
\begin{equation}\label{eq: hodge}
b_1=2q,\quad b_2=h^{0,2}+h^{2,0}+h^{1,1}\quad h^{p,q}=h^{q,p}=h^{2-p,2-q}, p,q=0,1,2.
\end{equation}
We may see the equalities \eqref{eq: noether}, \eqref{eq: euler},\eqref{eq: hodge} as equations for the Hodge numbers $h^{p,q}$'s, assuming known $c_1,c_2,q$, and we have the following solution:
\begin{equation}\label{eq: hodge-chern}
\begin{cases}
h^{0,0}&=h^{2,2}=1,\\
h^{0,1}&=h^{1,0}=h^{1,2}=h^{2,1}=q,\\
h^{0,2}&=h^{2,0}=\frac{1}{12}(c_1^2+c_2)-(1-q),\\
h^{1,1}&=-\frac{1}{6}c_1^2+\frac{5}{6}c_2+2q.
\end{cases}
\end{equation}

\subsection{Computing Hodge numbers via Chern numbers}

In this section, we denote $\mathcal{A}$ a line arrangement in $\bb{P}^2$ and $\widetilde{F}$ be the associated surface. All the Hodge numbers $h^{p,q}$ and Chern numbers $c_1^2,c_2$ are for $\widetilde{F}$, namely, we abbreviate the notations $h^{p,q}(\widetilde{F})$ by $h^{p,q}$ etc.

The irregularity $q$ is closely related to the monodromy of $h^*: H^1(F)\to H^1(F)$, and indeed, $2q=\dim H^1(F)_{\neq 1}$ by \cite{DP11}

We first give the formulae for the Chern numbers of $\ov{F}$. In the first two examples below, $t_r\neq0$ only if $r|d$. By example \ref{ex: resolB}, we have
$$
DCI_{r,d}=-d(r-2)^2,\qquad DCII_{r,d}=d-(r-1)^2.
$$

\begin{ex}(Hesse arrangement) The Hesse arrangement is defined by
$$
Q=xyz((x^3+y^3+z^3)^3-27x^3y^3z^3)
$$
with $d=12$ with $t_2=12, t_4=9$.  Moreover, we have $q=3$, see \cite{BDS}.

For the Chern numbers, we first have by \eqref{eq: c1fbar} that$
K_{\ov{F}}^2=d(d-4)^2=768.
$
Since $DCI_{2,12}=0$ and $DCI_{4,12}=-48$, we have
\begin{eqnarray*}
c_1^2&=&K_{\ov{F}}^2+\sum_rt_rDCI_{r,d}
     =336.
\end{eqnarray*}
So by Proposition \ref{prop: criterion general1}, $\widetilde{F}$ is of general type.
Moreover, by \eqref{eq: c2fbar},
\begin{eqnarray*}
\chi(\ov{F})&=&d(d^2-4d+6)-(d-1)\sum_rt_r(r-1)^2
            =201.
\end{eqnarray*}
Since $DCII_{2,12}=11$ and $DCII_{4,12}=3$, so
\begin{eqnarray*}
c_2&=&\chi(\ov{F})+\sum_rt_rDCII_{r,d}=360.
\end{eqnarray*}
Finally, by Formula \eqref{eq: hodge-chern}, we obtain
$$
\begin{cases}
h^{0,0}&=h^{2,2}=1,\\
h^{0,1}&=h^{1,0}=h^{1,2}=h^{2,1}=q=3,\\
h^{0,2}&=h^{2,0}=\frac{1}{12}(c_1^2+c_2)-(1-q)=60,\\
h^{1,1}&=-\frac{1}{6}c_1^2+\frac{5}{6}c_2+2q=250.
\end{cases}
$$
In addition, the Chern ratio
$$
\frac{c_1^2}{c_2}=\frac{336}{360}=\frac{14}{15}.
$$
\end{ex}

\begin{ex} Consider the arrangement $\mathcal{A}(m,m,3)$ defined by
$$
Q=(x^m-y^m)(y^m-z^m)(z^m-x^m)=0.
$$
Then if $m=3$, we have $t_3=12$ and for $m\neq 3$, we have $t_3=m^2, t_m=3$. In addition,
if $m\equiv 0\mod 3$, then $q=2$, otherwise $q=1$, see \cite{Dpre1}.

Moreover, by Example \ref{ex: resolB}, the following hold:
$$
DCI_{3,d}=-d=-3m,\qquad DCII_{3,d}=d-4=3m-4
$$
and
$$
DCI_{m,d}=-d(m-2)^2=-3m(m-2)^2,\qquad DCII_{m,d}=d-(m-1)^2=3m-(m-1)^2.
$$
Therefore, by \eqref{eq: c1fbar},
\begin{eqnarray}
c_1^2&=&K_{\ov{F}}^2+\sum_rt_rDCI_{r,d}
  =3m(m-2)(5m-2).\label{eq: 33mc1}
\end{eqnarray}
and by \eqref{eq: c2fbar},
\begin{eqnarray}
c_2&=&\chi(\ov{F})+\sum_rt_rDCII_{r,d}
     =9m(m^2-2m+2).\label{eq: 33mc2}
\end{eqnarray}

\begin{enumerate}[(i)]
\item First consider the case where $m=2$. Then $q=1$ and we have from \eqref{eq: 33mc1},
$
c_1^2=0
$
and from \eqref{eq: 33mc2},
$
c_2=36.
$
Therefore, by Formula \eqref{eq: hodge-chern}, we have
$$
\begin{cases}
h^{0,0}&=h^{2,2}=1,\\
h^{0,1}&=h^{1,0}=h^{1,2}=h^{2,1}=q=1,\\
h^{0,2}&=h^{2,0}=\frac{1}{12}(c_1^2+c_2)-(1-q)=3,\\
h^{1,1}&=-\frac{1}{6}c_1^2+\frac{5}{6}c_2+2q=32.
\end{cases}
$$
\item Second, consider the case where $m=3$, then $q=2$ and
we have from \eqref{eq: 33mc1} that
$
c_1^2=117
$
implying that $\widetilde{F}$ is of general type by Proposition \ref{prop: criterion general1} and from \eqref{eq: 33mc2},
$
c_2=135.
$
Therefore, by Formula \eqref{eq: hodge-chern}, we have
$$
\begin{cases}
h^{0,0}&=h^{2,2}=1,\\
h^{0,1}&=h^{1,0}=h^{1,2}=h^{2,1}=q=2,\\
h^{0,2}&=h^{2,0}=\frac{1}{12}(c_1^2+c_2)-(1-q)=22,\\
h^{1,1}&=-\frac{1}{6}c_1^2+\frac{5}{6}c_2+2q=97.
\end{cases}
$$
\item Consider the case where $m>3$ and $m\not\equiv0\Mod3$. Then $q=1$. From \eqref{eq: 33mc1}, we get
$c_1^2=3m(m-2)(5m-2).$
Note that in our situation $m\geq 4$, hence $c_1^2\geq 3\cdot 4\cdot 2\cdot 18=432>9$, hence $\widetilde{F}$ is of general type by Proposition \ref{prop: criterion general1}.

In addition, from \eqref{eq: 33mc2} we have
$
c_2=9m(m^2-2m+2).
$
Therefore, by Formula \eqref{eq: hodge-chern}, we have
$$
\begin{cases}
h^{0,0}&=h^{2,2}=1,\\
h^{0,1}&=h^{1,0}=h^{1,2}=h^{2,1}=q=1,\\
h^{0,2}&=h^{2,0}=\frac{1}{12}(c_1^2+c_2)-(1-q)\\
       &=\frac{1}{2}m(m-1)(4m-5)\\
h^{1,1}&=-\frac{1}{6}c_1^2+\frac{5}{6}c_2+2q\\
       &=5m^3-9m^2+13m+2
\end{cases}\\
$$

\item Finally, we consider the case where $m>3$ and $m\equiv0\Mod3$. Then $q=2$ and form \eqref{eq: 33mc1}, we have
$$
c_1^2=3m(m-2)(5m-2)\geq 3\cdot 6\cdot(6-2)\cdot(5\cdot 6-2)>9,
$$
hence $\widetilde{F}$ is of general type by Proposition \ref{prop: criterion general1}. Moreover, from \eqref{eq: 33mc2}, we have $c_2=9m(m^2-2m+2)$, hence by Formula \eqref{eq: hodge-chern}, we have
$$
\begin{cases}
h^{0,0}&=h^{2,2}=1,\\
h^{0,1}&=h^{1,0}=h^{1,2}=h^{2,1}=q=2,\\
h^{0,2}&=h^{2,0}=\frac{1}{12}(c_1^2+c_2)-(1-q)\\
       &=\frac{1}{2}m(m-1)(4m-5)+1\\
h^{1,1}&=-\frac{1}{6}c_1^2+\frac{5}{6}c_2+2q\\
       &=5m^3-9m^2+13m+4
\end{cases}\\
$$
\end{enumerate}
Conclusion: for $m\geq 3$, the surface $\widetilde{F}$ is of general type. Furthermore, as $m\to\infty$, one can show that $h^{0,2}, h^{1,1}\to\infty$ while other Hodge numbers remains 1 or 2. In addition, the Chern ratio
$$
\frac{c_1^2}{c_2}=\frac{3m(m-2)(5m-2)}{9m(m^2-2m+2)}\to\frac{5}{3}\qquad\text{\rm as }\ m\to\infty.
$$
\end{ex}

\begin{ex}
Now we consider line arrangements which arise from restriction of higher dimensional hyperplane arrangements. The braid arrangement in $\bb{P}^n$ is given by
$$
\mathfrak{B}_n:\quad \prod_{0\leq i<j\leq n}(x_i-x_j)=0,
$$
consisting of $\binom{n+1}{2}$ hyperplanes. Let $E\subseteq\bb{P}^n$ be a \emph{generic} projective plane and let $\mathcal{A}_n=\mathfrak{B}_n|_E$ the restriction of $\mathfrak{B}_n$ to $E$. Then $\mathcal{A}_n$ is a line arrangements in the projective plane with only nodes and triple points such that
$$
d=\binom{n+1}{2}=\frac{n(n+1)}{2}
$$
and
$$
t_3=\binom{n+1}{3}.
$$
Indeed, any triple points of $\mathcal{A}_n$ corresponds to the intersection of exactly three hyperplanes in $\mathfrak{B}_n$, which is thus of the form $\{x_{i_1}=x_{i_2}=x_{i_3}\}$ for some $i_1<i_2<i_3$. Hence
$$
t_3=\#\{(i_1,i_2,i_3)\ :\  i_1,i_2,i_3\in[0,n],\ i_1<i_2<i_3\}=\binom{n+1}{3}.
$$
From Remark \ref{rk: d(d-1)2}, we have
$$
2t_2+6t_3=d(d-1)=\binom{n+1}{2}\biggl(\binom{n+1}{2}-1\biggr),
$$
hence
\begin{eqnarray*}
t_2&=&\frac {d(d-1)}{2}-3t_3
   =\frac{(n+1)n(n-1)(n-2)}{8}.
\end{eqnarray*}
Note that if $n\equiv1\Mod3$, then $d\equiv1\Mod 3$, otherwise $3|d$, so we consider the following two cases:
\begin{enumerate}[\rm(i)]
\item If $n\not\equiv1\Mod3$, we have $3|d$. Moreover, if $n=2,3$, then $q=1$, otherwise $q=0$ by \cite{MP}. In addition,
$$
DCI_{2,d}=0,\qquad DCII_{2,d}=d-1
$$
and
$$
DCII_{3,d}=-d,\qquad DCII_{3,d}=d-4.
$$
Hence,
\begin{eqnarray*}
c_1^2&=& K_{\ov{F}}^2+\sum_rt_rDCI_{r,d}
     =\frac{1}{24}n(n+1)(n-2)(n-3)(3n^2+19n+32),
\end{eqnarray*}
so if $n\geq 4$, $c_1^2>9$ and thus $\widetilde{F}$ is of general type by Proposition \ref{prop: criterion general1}. Moreover,
\begin{eqnarray*}
c_2&=&\chi(\ov{F})+\sum_rt_rDCII_{r,d}
   =\frac{1}{8}n(n+1)(n-2)(n^3+2n^2-3n-12).
\end{eqnarray*}
\item If $n\equiv1\Mod 3$, then $d\equiv1\Mod 3$ and $q=0$ by \cite{MP}. In addition,
$$
DCI_{2,d}=0,\qquad DCII_{2,d}=d-1
$$
and by Example \ref{ex: resolC}, we have
$$
DCII_{3,d}=-(d-1),\qquad DCII_{3,d}=d-1.
$$
Thus,
\begin{eqnarray*}
c_1^2&=& K_{\ov{F}}^2+\sum_rt_rDCI_{r,d}\\
     &=&\frac{1}{24}n(n+1)\biggl(3n^2(n^2-15)+2n(2n^2-21)+188\bigg)
\end{eqnarray*}
so if $n\geq 4$, $c_1^2>9$ and thus $\widetilde{F}$ is of general type by Proposition \ref{prop: criterion general1}. Moreover,
\begin{eqnarray*}
c_2&=&\chi(\ov{F})+\sum_rt_rDCII_{r,d}
   =\frac{1}{8}n(n+1)(n^4-7n^2-2n+20).
\end{eqnarray*}
\end{enumerate}
The concrete formulae for the Hodge numbers by applying \eqref{eq: hodge-chern} are left to the reader. For the Chern ratio, we have
$$
\lim_{n\to\infty}\frac{c_1^2}{c_2}=1.
$$
\end{ex}


\begin{thebibliography}{00}

\bibitem{BHPV} W.~Barth, K.~Hulek, C.~Peters, A.~Van~de~Ven, \emph{Compact Complex Surfaces}, Second enlarged edition, Springer 2004.


\bibitem{BDS} N.~Budur, A.~Dimca, M.~Saito, \emph{First Milnor cohomology of hyperplane arrangements}, in: ``Topology of Algebraic Varieties and Singularities'', Contemporary Mathematics 538(2011), 279--292.

\bibitem{Cat} F.~ Catanese, \emph{Kodaira fibrations and beyond: methods for moduli theory}, arXiv:1611.06617.

\bibitem{CS95} D.~Cohen, A.~Suciu, \emph{On Milnor fibrations of arrangements}, J. London Math. Soc. {\bf 51} (1) (1995), 105–-119.

\bibitem{Deb} O.~Debarre, \emph{Introduction to Mori theory}, notes from a course
http://www.math.ens.fr/~debarre/M2.pdf. 2010.

\bibitem{D92}A.~Dimca, \emph{Singularities and Topology of Hypersurfaces}, Springer-Verlag, 1992.

\bibitem{Dpre1} A.~Dimca, \emph{On the Milnor monodromy of the irreducible complex reflection arrangements}, arXiv:1606.04048.

\bibitem{DP11} A.~Dimca, S.~Papadima, \emph{Finite Galois covers, cohomology jump loci, formality properties, and multinets}, Annali Scuola Norm. Sup. Pisa {\bf 10}  (2011),  253--268.

\bibitem{DStgenMF} A.~Dimca, G.~Sticlaru, \emph{Computing the monodromy and pole order filtration on the Milnor fiber cohomology of plane curves}, arXiv:1609.06818.

\bibitem{EH} D.~Eisenbud, J.~Harris, 3264 and all that: Intersection Theory in Algebraic Geometry, upcoming book, online version.

\bibitem{Fu} W.~Fulton, \emph{Intersection theory}, Ergebnisse der Mathematik und ihrer Grenzgebiete (3) {\bf 2}, Springer, Berlin, 1984.

\bibitem{HI} F.~Hirzebruch, \emph{Arrangements of lines and algebraic surfaces}, Progress in Mathematics {\bf 36}, Birkh\"auser Boston, 1983, 113--140.


\bibitem{KN} S.~Kobayashi, K.~Nozimu, \emph{Foundations of differential geometry}, Volume {\bf II}, Interscience, Wiley, New York, 1969.

\bibitem{Ma} D.~Matei, \emph{On a class of algebraic surfaces constructed from arrangements of lines}, Revue Roumaine de Mathematiques Pures et Appliquees, {\bf 60} (2015), no.~3, 355--373.

\bibitem{MD1} D.~Mumford, \emph{The topology of normal singularities of an algebraic surface and a criterion for simplicity}, Inst. Hautes \'Etudes Sci. Publ. Math., No.~9 (1961), 5--22.

\bibitem{MP} A.~M{\rm $\check{a}$}cinic, S.~ Papadima, \emph{On the monodromy action on Milnor fibers of graphic arrangements}, Topology and its Applications {\bf 156} (2009), 761--774.

\bibitem{MPP}  A.~M{\rm $\check{a}$}cinic, S.~Papadima, C.~R.~Popescu, \emph{Modular equalities for complex reflexion arrangements}, arXiv:1406.7137.

\bibitem{Naie} D.~Naie,  \emph{The irregularity of cyclic multiple planes after Zariski}, Enseign. Math. (2) {\bf 53} (2007), 265--305.

\bibitem{OW} P.~Orlik, Ph.~Wagreich, \emph{Equivariant resolution of singularities with $\bb{C}^*$ action}, in: \emph{Proceedings of the Second Conference on Compact Transformation Groups II}, Lecture Notes in Mathematics 299, Springer-Verlag, Berlin, 1972, pp. 270--290.

\bibitem{PS14}  S.~Papadima, A.~Suciu, \emph{The Milnor fibration of a
hyperplane arrangement: from modular resonance to algebraic
monodromy},  arXiv:1401.0868.

\bibitem{Re} M.~Reid, \emph{The Du Val singularities $A_n, D_n, E_6, E_7, E_8$}, online notes.

\bibitem{Rou} X.~Roulleau, \emph{Divisor arrangements and algebraic surfaces}, Commentarii Mathematici Universitatis Sancti Pauli {\bf 57}, No.~1, 2008.

\bibitem{Sa} F. Sakai, \emph{Semi-stable curves on algebraic surfaces and logarithmic pluricanonical maps},
Math. Ann. {\bf 254}, p. 89--120, 1980.

\bibitem{S1} A.~Suciu,
\emph{ Fundamental groups of line arrangements: Enumerative aspects},
 	Contemporary Math., Amer. Math. Soc., {\bf 276} (2001), 43--79.
 	
\bibitem{S4} A.~Suciu,
\emph{Hyperplane arrangements and Milnor fibrations},
 Ann. Fac. Sci. Toulouse Math. (6) {\bf 23} (2014),  417--481.

\bibitem{Tr} P. Tretkoff, \emph{Complex Ball Quotients and Line Arrangements in the Projective Plane (MN-51)}, With an appendix by Hans-Christoph Im Hof, Priceton University Press, 2016.


\bibitem{Yau} S.~T.~Yau, \emph{Calabi's Conjecture and some new results in algebraic geometry}, Proc. Nat. Acad. Sci. USA, Volume {\bf 74}, No.~5 (1977), 1798--1799.

\end{thebibliography}
\end{document}